\documentclass[12pt, reqno]{article}
\usepackage{amssymb, mathrsfs}
\usepackage{latexsym}
\usepackage{amsmath}
\usepackage{amsthm}
\usepackage{amsfonts}
\usepackage{dsfont}
\usepackage{mathtools}
\usepackage{epsfig}
\usepackage{verbatim}
\usepackage[all,cmtip]{xy}
\usepackage[left=1.3in,right=1.3in,top=1in,bottom=1in]{geometry}
\usepackage{enumerate}
\usepackage[colorlinks=true,linkcolor=blue,citecolor=blue]{hyperref}

\usepackage{tikz}
\usetikzlibrary{matrix}
\usetikzlibrary{arrows}

\usepackage{authblk}

\newtheorem*{lemma*}{Lemma}

\newtheorem*{theorem*}{Theorem}
\newtheorem*{conjecture*}{Conjecture}

\newtheorem{theorem}{Theorem}[section]
\newtheorem{proposition}[theorem]{Proposition}
\newtheorem{lemma}[theorem]{Lemma}
\newtheorem*{thma}{Theorem A}
\newtheorem*{thmb}{Theorem B}
\newtheorem{corollary}[theorem]{Corollary}
\newtheorem{conjecture}[theorem]{Conjecture}

\newtheorem*{claim*}{Claim}

\theoremstyle{definition}
\newtheorem{definition}[theorem]{Definition}

\newtheorem*{question*}{Question}

\theoremstyle{remark}
\newtheorem{remark}[theorem]{Remark}

\newcommand{\ind}{\textup{Ind}}
\newcommand{\supp}{\textup{Supp}}
\newcommand{\ev}{\textup{ev}}

\newcommand{\fin}{\textup{fin}}
\newcommand{\rfK}{\widetilde K}
\newcommand{\ord}{\textup{order}}

\newcommand{\diff}{\textup{Diff}}

\begin{document}
\title{Higher rho invariants and the moduli space of positive scalar curvature metrics}

\author{Zhizhang Xie\thanks{Email: \texttt{xie@math.tamu.edu}; partially supported by AMS-Simons Travel Grant.} }
\author{Guoliang Yu\thanks{Email: \texttt{guoliangyu@math.tamu.edu}; partially supported by the US National Science Foundation.}}
\affil{Department of Mathematics, Texas A\&M University}
\date{}

\maketitle

\begin{abstract}
Given a closed smooth manifold $M$ which carries a positive scalar curvature metric, one can associate an abelian group $P(M)$ to the space of positive scalar curvature metrics on this manifold. The group of all diffeomorphisms of the manifold naturally acts on $P(M)$. The moduli group $\widetilde P(M)$ of positive scalar curvature metrics  is defined to be the quotient abelian group of this action, i.e. the coinvariant of the action. $\widetilde P(M)$ measures the size of the moduli space of positive scalar curvature metrics on $M$.
 In this paper, we use the higher rho invariant and the finite part of the $K$-theory of the group $C^\ast$-algebra of $\pi_1(M)$ to give a lower bound of the rank of the moduli group $\widetilde P(M)$.
\end{abstract}

\section{Introduction}

Let $M$ be a closed smooth manifold. Suppose $M$ carries a metric of positive scalar curvature. It is well known that the space of all Rimennian metrics on $M$ is contractible, hence topologically trivial. To the contrary, the space of all positive scalar curvature metrics on $M$, denoted by $\mathcal R^+(M)$, often has very nontrivial topology. For example, the homotopy groups of $\mathcal R^+(M)$ often contain many nontrivial elements (cf. \cite{Hanke:2012fk, NH74}). In particular, $\mathcal R^+(M)$ is often not connected and in fact has infinitely many connected components (cf. \cite{ MR1339924, ELPP01, MR2366359, MR1268192}). For example, by using the Cheeger-Gromov $L^2$-rho invariant \cite{MR806699} and the delocalized eta invariant of Lott \cite{MR1726745},  Piazza and Schick showed that $\mathcal R^+(M)$ has finitely many connected components, if $M$ is a closed spin manifold with $\dim M = 4k+3 \geq 5$ and $\pi_1(M)$ contains torsion \cite[Theorem 1.3]{MR2366359}. 

Following Stolz \cite{stolz-concordance, SS95}, Weinberger and Yu introduced an abelian group $P(M)$ to measure the size of the space of positive scalar curvature metrics \cite{Weinberger:2013qy}. In addition,  they used the finite part of $K$-theory of the maximal group $C^\ast$-algebra $C^\ast_{\max}(\pi_1(M))$ to give a lower bound of the rank of $P(M)$. A special case of their theorem states that the rank of $P(M)$ is  $\geq 1$, if $M$ is a closed spin manifold with $\dim M = 2k+1\geq 5$ and $\pi_1(M)$ contains torsion \cite[Theorem 4.1]{Weinberger:2013qy}. In particular, this implies the above theorem of Piazza and Schick. 

In this paper, inspired by the results of Piazza and Schick \cite{MR2366359} and Weinberger and Yu \cite{Weinberger:2013qy}, we use the finite part of $K$-theory of $C^\ast_{\max}(\pi_1(M))$ to study the moduli space of positive scalar curvature metrics.  Recall that the group of diffeomorphisms on $M$, denoted by $\diff(M)$, acts on $\mathcal R^+(M)$ by pulling back the metrics. The moduli space of positive scalar curvature metrics is defined to be the quotient space $\mathcal R^+(M)/\diff(M)$. Similarly, $\diff(M)$ acts on the group $P(M)$ and we denote the coinvariant of the action by $\widetilde P(M)$. That is, $\widetilde P(M) = P(M)/P_0(M)$ where $P_0(M)$ is the subgroup of $P(M)$ generated by elements of the form $[x] - \psi^\ast[x]$ for all $ [x]\in P(M)$ and all $\psi\in \diff(M)$. We call $\widetilde P(M)$ the moduli group of positive scalar curvature metrics on $M$. It measures the size of the moduli space of positive scalar curvature metrics on $M$. The following conjecture gives a lower bound for the rank of the abelian group $\widetilde{P}(M)$.
\begin{conjecture*}
Let $M$ be a closed spin manifold  with  $\pi_1(M) = \Gamma$ and $\dim M = 2k+1 \geq 5$,  which carries a positive scalar curvature metric.  Then the rank of the abelian group $\widetilde P(M)$ is  $\geq N_\fin(\Gamma)$, where $N_\fin(\Gamma)$ is the cardinality of the following collection of positive integers: 
\[ \big\{ d \in \mathbb N_+ \mid \exists \gamma\in \Gamma \ \textup{such that}\ \ord(\gamma) = d \ \textup{and} \ \gamma\neq e\big\}. \]
\end{conjecture*}

In the main theorem (Theorem $\ref{thm:psc}$) of the paper, we show that the above conjecture holds for strongly finitely embeddable groups (see Section $\ref{sec:sfe}$ for the definition of strongly finite embeddability). 
\begin{thma}
Let $M$ be a closed spin manifold which carries a positive scalar curvature metric with $\dim M = 2k+1 \geq 5$. If the fundamental group $\Gamma = \pi_1(M)$ of $M$ is strongly finitely embeddable into Hilbert space, then the rank of the abelian group $\widetilde P(M)$ is $\geq N_\fin(\Gamma).$ 
\end{thma}

To prove this theorem, we need invariants under the action of the diffeomorphism group that come from Dirac operators to distinguish positive scalar curvature metrics. The index theoretic techniques used in \cite{Weinberger:2013qy}, for example,  do not produce such invariants. Our proof uses the higher rho invariant; this is a secondary invariant associated to Dirac operators and in particular depends on the choice of Riemannian metric. We show that the higher rho invariant remains unchanged in a certain $K$-theory group under the action of the diffeomorphism group, allowing us to distinguish elements in $\widetilde P(M)$. This is  the main novelty of the proof.  

The concept of strongly finite embeddability into Hilbert space for groups is a stronger version of the notion of finite embeddability into Hilbert space, which was introduced by Weinberger and Yu in \cite{Weinberger:2013qy}.
We refer the reader to \cite{Weinberger:2013qy} or Section $\ref{sec:fpk} \ \&\ \ref{sec:sfe}$ below for the precise definitions. The class of groups that are strongly finitely embeddable into Hilbert space includes all residually finite groups, amenable groups, hyperbolic groups, and virtually torsion free groups (e.g. $Out(F_n)$). By definition, strongly finite embeddability implies finite embeddability (see Section $\ref{sec:fpk} \ \&\ \ref{sec:sfe}$).   It is an open question whether the converse is true, that is, whether finite embeddability implies strongly finite embeddability.

For general groups, we prove the following weaker version of the conjecture. This result is motivated by a theorem of Piazza and Schick \cite{MR2366359}.  They used a different method to show that the moduli space $\mathcal R^+(M)/\diff(M)$ has infinitely many connected components when $\dim M = 4k+3 \geq 5$ and the fundamental group $\pi_1(M)$ is not torsion free \cite{MR2366359}. 

\begin{thmb}
Let $M$ be a closed spin manifold which carries a positive scalar curvature metric with $\dim M = 2k+1 \geq 5$. If $\Gamma = \pi_1(M)$ is not torsion free, then the rank of the abelian group $\widetilde P(M)$ is $\geq 1$.
\end{thmb}

The paper is organized as follows. In Section $\ref{sec:pre}$, we recall the definition of the index map, the local index map and the higher rho invariant. In Section $\ref{sec:fpk}$, we discuss the finite part of $K$-theory of group $C^\ast$-algebras. In Section $\ref{sec:sfe}$, we introduce the notion of strongly finite embeddability into Hilbert space for groups. We prove the main results of the paper in Section $\ref{sec:psc}$.

\section{Preliminaries}\label{sec:pre}
In this section, we briefly recall some standard definitions. We refer the reader to \cite{MR2431253, MR1147350, MR1451759} for more details.

\subsection{Maximal Roe algebras and localization algebras}
Let $X$ be a proper metric space. That is, every closed ball in $X$ is compact. An	$X$-module is a separable Hilbert space equipped with a	$\ast$-representation of $C_0(X)$, the algebra of all continuous functions on $X$ which vanish at infinity. An	$X$-module is called nondegenerate if the $\ast$-representation of $C_0(X)$ is nondegenerate. An $X$-module is said to be ample if no nonzero function in $C_0(X)$ acts as a compact operator. Throughout the paper, we only work with ample $X$-modules.  
\begin{definition}
Let $H_X$ be a $X$-module and $T$ be a bounded linear operator acting on $H_X$. 
\begin{enumerate}[(i)]
\item The propagation of $T$ is defined to be $\sup\{ d(x, y)\mid (x, y)\in \supp(T)\}$, where $\supp(T)$ is  the complement (in $X\times X$) of the set of  points $(x, y)\in X\times X$ for which there exist $f, g\in C_0(X)$ such that $gTf= 0$ and $f(x)\neq 0$, $g(y) \neq 0$;
\item $T$ is said to be locally compact if $fT$ and $Tf$ are compact for all $f\in C_0(X)$; 
\end{enumerate}
\end{definition}

\begin{definition}
Let $H_X$ be a nondegenerate ample $X$-module and $\mathcal B(H_X)$ the set of all bounded linear operators on $H_X$.  The (reduced) Roe algebra of $X$, denoted by $C_{r}^\ast(X)$, is the $C^\ast$-algebra generated by all locally compact operators with finite propagations in $\mathcal B(H_X)$.

\end{definition}

In this paper, we will mainly work with the maximal Roe algebra \cite{MR2431253}. Let us denote by $\mathbb C[X]$ the $\ast$-algebra of  all locally compact operators with finite propagations in $\mathcal B(H_X)$. 
\begin{lemma}[{\cite[Lemma 3.4]{MR2431253}}]
For each $a\in \mathbb C[X]$, there is a constant $c_a$ such that 
\[  \|\phi(a) \| \leq c_a\]
for all $\ast$-representation $\phi: \mathbb C[X] \to \mathcal B(H)$, where $H$ is a Hilbert space. 
\end{lemma} 
We refer the reader to \cite[Lemma 3.4]{MR2431253} for a detailed proof of the above lemma.
\begin{definition}
 For each $a\in \mathbb C[X]$, we define 
\[  \|a\|_{\max} = \sup_{\phi} \ \big\{\|\phi(a)\| \mid \phi: \mathbb C[X] \to \mathcal B(H) \ \textup{a $\ast$-representation} \big\}, \]
where $\phi$ runs through all $\ast$-representations of $\mathbb C[X]$.
The maximal Roe algebra $C^\ast_{\max} (X)$ is defined to be the completion of $\mathbb C[X]$ with respect to this maximal norm. 
\end{definition}

For notational simplicity, we write $C^\ast(X) = C^\ast_{\max}(X) $ from now on. Let us also recall the definition of localization algebras \cite{MR1451759}. Again, we work with the maximal version. 
\begin{definition}

\begin{enumerate}[(i)]
\item $C_L^\ast(X)$ is the $C^\ast$-algebra generated by all bounded and uniformly continuous functions $f: [0, \infty) \to C^\ast(X)$  such that 
\[ \textup{propagation of $f(t) \to 0 $, as $t\to \infty$.}  \]
\item $C_{L, 0}^\ast(X)$ is the kernel of the evaluation map 
\[  \ev : C_L^\ast(X) \to C^\ast(X),  \quad   \ev (f) = f(0).\]
In particular, $C_{L, 0}^\ast(X)$ is a closed  ideal of $C_L^\ast(X)$.
\end{enumerate}
\end{definition}

Now suppose there is a countable discrete group $\Gamma$ acting on $X$ properly and cocompactly. Let $H_X$ be an $X$-module equipped with a covariant unitary representation of $\Gamma$. If we denote the representation of $C_0(X)$ by $\varphi$ and the representation of $\Gamma$ by $\pi$, this means 
\[  \pi(\gamma) (\varphi(f) v )  =  \varphi(f^\gamma) (\pi(\gamma) v),\] 
where $f\in C_0(X)$, $\gamma\in \Gamma$, $v\in H_X$ and $f^\gamma (x) = f (\gamma^{-1} x)$. In this case, we call $(H_X, \Gamma, \varphi)$ a covariant system. 

\begin{definition}[\cite{MR2732068}]
A covariant system $(H_X, \Gamma, \varphi)$ is called admissible if 
\begin{enumerate}[(1)]
\item the $\Gamma$-action on $X$ is proper and cocompact;
\item $H_X$ is a nondegenerate ample $X$-module;
\item for each $x\in X$, the stabilizer group $\Gamma_x$ acts on $H_X$ regularly in the sense that the action is isomorphic to the action of $\Gamma_x$ on $l^2(\Gamma_x)\otimes H$ for some infinite dimensional Hilbert space $H$. Here $\Gamma_x$ acts on $l^2(\Gamma_x)$ by translations and acts on $H$ trivially. 
\end{enumerate}
\end{definition}
We remark that for each proper metric space $X$ with a proper and cocompact isometric action of $\Gamma$, there exists an admissible covariant system $(H_X, \Gamma, \varphi)$. Moreover,  if $\Gamma$ acts freely on $X$, then the condition $(3)$ is automatically satisfied.  If no confusion arises, we will simply denote an admissible covariant system $(H_X, \Gamma, \varphi)$ by $H_X$ and call it an admissible $(X, \Gamma)$-module.

\begin{definition}
Let $X$ be a locally compact metric space $X$ with a proper and cocompact isometric action of $\Gamma$. Suppose $H_X$ is an admissible $(X, \Gamma)$-module. We denote by $\mathbb C[X]^\Gamma$ the $\ast$-algebra of all $\Gamma$-invariant locally compact operators with finite propagations in $\mathcal B(H_{X})$. We define the maximal $\Gamma$-invariant Roe algebra $C^\ast(X)^\Gamma = C^\ast_{\max}(X)^\Gamma$ to be the completion of $\mathbb C[X]^\Gamma$ under the maximal norm:
\[  \|a\|_{\max} = \sup_{\phi} \ \big\{\|\phi(a)\| \mid \phi: \mathbb C[X]^\Gamma \to \mathcal B(H) \ \textup{a $\ast$-representation} \big\}. \]
We can also define the maximal $C_L^\ast(X)^\Gamma$ and $C_{L,0}^\ast(X)^\Gamma$  similarly. 
\end{definition}

\begin{remark}
In fact,  $\mathbb C[X]^\Gamma$ is isomorphic to $\mathcal K[\Gamma]$ the group algebra over the coefficient ring $\mathcal K$, where $\mathcal K$ is the algebra of all compact operators. In particular, $C_{\max}^\ast(X)^\Gamma \cong C_{\max}^\ast(\Gamma) \otimes \mathcal K$, where $C_{\max}^\ast(\Gamma)$ is the maximal group $C^\ast$-algebra of $\Gamma$.  
\end{remark}

\subsection{K-homology}
In this subsection, we recall the definition of K-homology due to Kasparov. 
Let $X$ be a locally compact metric space with a proper and cocompact isometric action of $\Gamma$. The $K$-homology groups $K_j^\Gamma(X)$, $j=0, 1$, are generated by the following cycles modulo certain equivalence relations (cf. \cite{GK88}):
\begin{enumerate}[(i)]
\item an even cycle for $K_0^\Gamma(X)$ is a pair $(H_X, F)$, where $H_X$ is an admissible $(X, \Gamma)$-module and $F\in \mathcal B(H_X)$ such that $F$ is $\Gamma$-equivariant,     $F^\ast F - I$ and $F F^\ast - I$ are locally compact and $[F, f] = F f - fF $ is compact for all $f\in C_0(X)$.   
\item an odd cycle for $K_1^\Gamma(X)$ is a pair $(H_X, F)$, where $H_X$ is an admissible  $(X, \Gamma)$-module and $F$ is a $\Gamma$-equivariant self-adjoint operator in $\mathcal B(H_X)$ such that $F^2 - I$ is locally compact and $[F, f]$ is compact for all $f\in C_0(X)$. 
\end{enumerate}


\begin{remark}
In the general case where the action of $\Gamma$ on $X$ is not necessarily cocompact, we define 
\[  K_i^\Gamma(X)  = \varinjlim_{Y\subseteq X}  K_i^\Gamma(Y) \]
where $Y$ runs through all closed $\Gamma$-invariant subsets of $X$ such that $Y/\Gamma$ is compact. 
\end{remark}

\subsection{K-theory and index maps}\label{sec:kt}
In this subsection, we recall the standard construction of the index maps in $K$-theory of $C^\ast$-algebras. 
For a short exact sequence of $C^\ast$-algebras $ 0 \to  \mathcal J \to \mathcal A \to \mathcal A/\mathcal J \to 0,$
we have a six-term exact sequence in $K$-theory:
\[ 
\xymatrix { K_0( \mathcal J ) \ar[r] & K_0(\mathcal A) \ar[r] & K_0(\mathcal A/\mathcal J )  \ar[d]^{\partial_1} \\
K_1(\mathcal A/\mathcal J) \ar[u]^{\partial_0} 
  & K_1(\mathcal A) \ar[l] & K_1(\mathcal J) \ar[l]
}
\]  
Let us recall the definition of the index maps $\partial_i$.

\begin{enumerate}[(1)]
\item Even case. Let $u$ be an invertible element in $\mathcal A/\mathcal J$. Let $v$ be the inverse of $u$ in $\mathcal A/\mathcal J$. Now suppose $U, V\in \mathcal A$ are lifts of $u$ and $v$. We define 
\[ W = \begin{pmatrix} 1 & U\\ 0 & 1\end{pmatrix} \begin{pmatrix} 1 & 0 \\ -V & 1\end{pmatrix} \begin{pmatrix} 1 & U  \\ 0 & 1 \end{pmatrix}\begin{pmatrix} 0 & -1\\ 1 & 0 \end{pmatrix}. \]
Notice that $W$ is invertible and a direct computation shows that 
\[  W - \begin{pmatrix} U & 0 \\ 0 & V \end{pmatrix} \in   \mathcal J.\]
Consider the idempotent
\begin{equation}\label{eq:keven}
 P = W \begin{pmatrix} 1 & 0 \\ 0 & 0\end{pmatrix} W^{-1} = \begin{pmatrix} UV + UV(1-UV) & (2 + UV)(1-UV) U \\ V(1-UV) & (1-UV)^2\end{pmatrix}.
\end{equation}
We have  
\[ P - \begin{pmatrix} 1 & 0 \\0 & 0\end{pmatrix} \in \mathcal J. \]
By definition, 
\[  \partial_0 ([u]) := [P] - \left[\begin{pmatrix} 1 & 0 \\0 & 0\end{pmatrix} \right] \in K_0(\mathcal J).\]
\item Odd case. Let $q $ be an idempotent in $\mathcal A/\mathcal J$ and $Q$ a lift of $q$ in $\mathcal A$. Then 
\[  \partial_1([q])  := [e^{2\pi iQ}] \in K_1(\mathcal J). \] 

\end{enumerate}

\subsection{Higher index map and local index map}\label{sec:localind}
In this subsection, we recall the construction of the higher index  map and the local index map \cite{MR1451759, MR2732068}. 

Let $(H_X, F)$ be an even cycle for $K_0^\Gamma(X)$. Choose  a $\Gamma$-invariant locally finite open cover $\{U_i\}$ of $X$ with diameter$(U_i) < c$ for some fixed $c > 0$. Let $\{\phi_i\} $ be a $\Gamma$-invariant continuous partition of unity subordinate to $\{U_i\}$. We define
\[  \mathcal F = \sum_{i} \phi^{1/2}_i F \phi^{1/2}_i,\]
where the sum converges in strong operator norm topology. It is not difficult to see that $(H_X, \mathcal F)$ is equivalent to $(H_X, F)$ in $K^\Gamma_0(X)$. By using the fact that $\mathcal F$ has finite propagation, we see that $\mathcal F$ is a multiplier of $C^\ast(X)^\Gamma$ and, in fact, is a unitary modulo $C^\ast(X)^\Gamma$. Consider the short exact sequence of $C^\ast$-algebras 
\[ 0 \to C^\ast(X)^\Gamma \to \mathcal M(C^\ast(X)^\Gamma) \to \mathcal M(C^\ast(X)^\Gamma) /C^\ast(X)^\Gamma \to 0 \]
where $\mathcal M(C^\ast(X)^\Gamma)$ is the multiplier algebra of $C^\ast(X)^\Gamma$.
By the construction in Section $\ref{sec:kt}$ above, $\mathcal F$ produces a class $[\mathcal F] \in K_0(C^\ast(X)^\Gamma)$. We define the higher index of $(H_X, F)$ to be $[\mathcal F]$. From now on, we denote $[\mathcal F]$ by  $\ind(H_X, F)$ or simply $\ind(F)$, if no confusion arises.

To define the local index class of $(H_X, F)$, we need to use a family of partitions of unity. More precisely, for each $n\in \mathbb N$, let  $\{U_{n, j}\}$ be a $\Gamma$-invariant locally finite open cover of $X$ with diameter $(U_{n,j}) < 1/n$ and $\{\phi_{n, j}\}$ be a $\Gamma$-invariant continuous partition of unity subordinate to $\{U_{n, j}\}$. We define 
\begin{equation}\label{eq:path}
 \mathcal F(t) = \sum_{j} (1 - (t-n)) \phi_{n, j}^{1/2} \mathcal F \phi_{n,j}^{1/2} + (t-n) \phi_{n+1, j}^{1/2} \mathcal F \phi_{n+1, j}^{1/2}
\end{equation}
for $t\in [n, n+1]$. 

Then $\mathcal F(t), 0\leq t <\infty,$ is a multiplier of $C^\ast_L(X)^\Gamma$ and a unitary modulo $C^\ast_L(X)^\Gamma$. Hence by the construction in Section $\ref{sec:kt}$ above, $\mathcal F(t),  0\leq t <\infty,$ gives rise to an element in $K_0(C^\ast_L(X)^\Gamma)$. We call this class $[\mathcal F(t)]\in K_0(C^\ast_L(X)^\Gamma)$  the local index of $(H_X, F)$. If no confusion arises, we denote this local index class by $\ind_L(H_X, F)$ or simply $\ind_L(F)$.

Now let $(H_X, F)$ be an odd cycle in $K_1^\Gamma(X)$. With the same notation from above, we set $q = \frac{\mathcal F + 1}{2}$. Then the index class of $(H_X, F)$ is defined to be $ [e^{2\pi i q}]\in K_1(C^\ast(X)^\Gamma).$ For the local index class of $(H_X, F)$, we use $q(t) = \frac{\mathcal F(t) + 1}{2}$ in place of $q$.

\subsection{Higher rho invariant}\label{sec:rho}
In this subsection, we define the higher rho invariant \cite[Definition 7.1]{MR2761858}. We carry out the construction in the odd dimensional case. The even dimensional case is similar.

Suppose $X$ is an odd dimensional complete spin manifold without boundary and we fix a spin structure on $X$. Assume that there is a discrete group $\Gamma$ acting on $M$  properly and cocompactly by isometries. In addition, we assume the action of $\Gamma$ preserves the spin structure on $X$. A typical such example comes from a Galois cover $\widetilde M$ of a closed spin manifold $M$ with $\Gamma$ being the group of deck transformations.

Let $S$ be the spinor bundle over $X$ and $D= D_{X}$ be the associated Dirac operator on $X$.  Let $H_X = L^2(X, S) $ and  
\[  F = D( D^2 + 1)^{-1/2}.  \]
Then $(H_X, F)$ defines a class in $K_1^\Gamma(X)$. Note that $F$ lies in the multiplier algebra of $C^\ast(X)^\Gamma$, since $F$ can be approximated by elements of finite propagation in the multiplier algebra of $C^\ast(X)^\Gamma$. As a result, we can directly work with\footnote{In other words, there is no need to pass to the operator $\mathcal F$ or $\mathcal F(t)$ as in the general case.} 
\begin{equation}\label{eq:path2}
 F(t) = \sum_{j} (1 - (t-n)) \phi_{n, j}^{1/2} F \phi_{n,j}^{1/2} + (t-n) \phi_{n+1, j}^{1/2} F \phi_{n+1, j}^{1/2}
\end{equation}
for $t\in [n, n+1]$. And the same argument as before defines the index class and the local index class of $(H_X, F)$. We shall denote them by $\ind(D)\in K_1(C^\ast(X)^\Gamma)$ and $\ind_L(D)\in K_1(C_L^\ast(X)^\Gamma)$ respectively.

Now suppose in addition $X$ is endowed with a complete Riemannian metric $g$ whose scalar curvature $\kappa$ is positive everywhere, then the associated Dirac operator in fact naturally defines a class in $K_1(C_{L, 0}^\ast(X)^\Gamma)$. Indeed, recall that
\[ D^2 = \nabla^\ast \nabla + \frac{\kappa}{4}, \]
where $\nabla: C^\infty(X, S) \to C^\infty(X, T^\ast X\otimes S)$ is the associated connection and $\nabla^\ast$ is the adjoint of $\nabla$. If $\kappa >0$, then it follows immediately that $ D$ is invertible. So, instead of $D(D^2+1)^{-1/2}$, we can use 
\[  F := D|D|^{-1}. \]
Note that $\frac{F+1 }{2}$ is a genuine projection. Define $F(t)$ as in formula $\eqref{eq:path2}$, and define $q(t) := \frac{F(t) + 1}{2}$. By the construction in Section $\ref{sec:kt}$,  we form the path of unitaries $u(t) = e^{2\pi i q(t)}, 0\leq t < \infty$, in $ (C_L^\ast(X)^\Gamma)^+$. Notice that $u(0) = 1$. So  this path $u(t), 0\leq t < \infty,$ in fact lies in $(C_{L,0}^\ast(X)^\Gamma)^+$, therefore defines a class in $K_1(C_{L, 0}^\ast(X)^\Gamma)$. 

Let us now define the higher rho invariant. It was first introduced by Higson and Roe \cite[Definition 7.1]{MR2761858}. Our formulation is slightly different from that of Higson and Roe. The equivalence of the two definitions was shown in \cite[Section 6]{Xie2014823}. 
\begin{definition}
The higher rho invariant $\rho(D, g)$  of the pair $(D, g)$ is defined to be the $K$-theory class $[u(t)]\in K_1(C_{L, 0}^\ast(X)^\Gamma).$
\end{definition}
The even dimensional case is similar, where one needs to work with the natural $\mathbb Z/2\mathbb Z$-grading on the spinor bundle.

\section{Finite part of $K$-theory for group $C^\ast$-algebras}\label{sec:fpk}
In this section, we discuss the finite part of $K$-theory for group $C^\ast$-algebras. We also review some results from the paper of Weinberger and Yu \cite{Weinberger:2013qy}.

Let $\Gamma$ be a countable discrete group. Recall that an element $\gamma\in \Gamma$ is said to have finite order $d$ if $d$ is the smallest positive integer such that $\gamma^d = e$, where $e$ is the identity element of $\Gamma$. In this case, we write $\ord(\gamma) = d$. 

Let $\gamma\in \Gamma$ be an element of finite order $d$.  We define
\[ p_\gamma = \frac{1}{d}\sum_{k=1}^d \gamma^k. \] 
We observe that $p_\gamma$ is an idempotent of the group algebra $\mathbb C[\Gamma]$. Recall that $C^\ast(\Gamma)$ is the maximal group $C^\ast$-algebra of $\Gamma$. We define  $K_0^\fin(C^\ast(\Gamma))$ to be the abelian subgroup of $K_0(C^\ast(\Gamma))$ generated by $[p_\gamma]$ for all finite order elements $\gamma\neq e$ in $\Gamma$.  We call $K_0^\fin(C^\ast(\Gamma))$ the finite part of $K_0(C^\ast(\Gamma))$.

\begin{definition} Let $\{\gamma_1, \cdots, \gamma_n\}$ be a collection of nontrivial elements (i.e. $\gamma_i\neq e$) with finite order in $\Gamma$.
We define $\mathcal M_{\gamma_1, \cdots, \gamma_n}$ to be the abelian subgroup of $K_0^\fin(C^\ast(\Gamma))$ generated by $\{ [p_{\gamma_1}], \cdots, [p_{\gamma_n}]\}$.
\end{definition}

Weinberger and Yu made the following conjecture in \cite{Weinberger:2013qy}. The conjecture gives a lower bound of the rank of the abelian group $K^\fin_0(C^\ast(\Gamma)$. 

\begin{conjecture}[\cite{Weinberger:2013qy}] \label{wyconj}
Suppose $\{\gamma_1, \cdots, \gamma_n\}$ is a collection of elements in $\Gamma$ with distinct finite orders and $\gamma_i\neq e$ for all $1\leq i\leq n$. Then 
\begin{enumerate}[(1)]
\item  the abelian group $\mathcal M_{\gamma_1, \cdots, \gamma_n}$ has rank $n$,
\item any nonzero element in $\mathcal M_{\gamma_1, \cdots, \gamma_n}$ is not in the image of the assembly map 
\[ \mu: K_0^\Gamma(E\Gamma) \to K_0(C^\ast(\Gamma)), \]
where $E\Gamma$ is the universal space for proper and free $\Gamma$-action.
\end{enumerate}
\end{conjecture}

They also proved that the conjecture holds for a large class of groups, including all residually finite groups, amenable groups, Gromov's monster groups, virtually torsion free groups (e.g. $Out(F_n)$), and any group of analytic diffeomorphisms of a connected analytic manifold fixing a given point \cite{Weinberger:2013qy}. In fact, they introduced a notion of finite embeddability for groups and showed that the conjecture holds for all groups that are finitely embeddable into Hilbert space. Then they verified that the list of groups mentioned above are finitely embeddable into Hilbert space.

Let us recall the notion of finite embeddability for groups in the following. We shall first recall the notion of coarse embeddability due to Gromov.  
\begin{definition}[Gromov]
A countable discrete group $\Gamma$ is said to be coarsely embeddable into Hilbert space $H$ if there exists a map $f:\Gamma \to H$ such that 
\begin{enumerate}[(1)]
\item for any finite subset $F\subseteq \Gamma$, there exists $R >0$ such that if $\gamma^{-1}\beta\in F$, then $\|f(\gamma)- f(\beta)\|\leq R$;
\item for any $S>0$, there exists a finite subset $F\subseteq \Gamma$ such that if $\gamma^{-1}\beta\in \Gamma-F$, then $\|f(\gamma) - f(\beta)\| \geq S$.
\end{enumerate}
\end{definition}

The notion of finite embeddability for groups, introduced by Weinberger and Yu, is more flexible than the notion of coarse embeddability. 

\begin{definition}[\cite{Weinberger:2013qy}]
A countable discrete group $\Gamma$ is said to be finitely embeddable into Hilbert space $H$ if for any finite subset $F\subseteq \Gamma$, there exist a group $\Gamma'$ that is coarsely embeddable into $H$ and  a map $\phi: F\to \Gamma'$ such that 
\begin{enumerate}[(1)]
\item if $\gamma, \beta$ and $\gamma\beta$ are all in $F$, then $\phi(\gamma\beta) = \phi(\gamma) \phi(\beta)$;
\item if $\gamma$ is a finite order element in $F$, then $\ord(\phi(\gamma)) = \ord(\gamma)$. 
\end{enumerate}
\end{definition}

As mentioned above, Weinberger and Yu proved that Conjecture $\ref{wyconj}$ holds for all groups that are finitely embeddable into Hilbert space \cite[Theorem 1.4]{Weinberger:2013qy}. This result has some interesting applications, and we list a few of them in the following. We denote by  $N_\fin(\Gamma)$ the cardinality of the following collection of positive integers: 
\[ \big\{ d \in \mathbb N_+ \mid \exists \gamma\in \Gamma \ \textup{such that}\ \ord(\gamma) = d \ \textup{and} \ \gamma\neq e\big\}. \]
Then we have the following results from \cite{Weinberger:2013qy}.
\begin{enumerate}[(i)]
\item Let $\Gamma$ be an arbitrary countable discrete group. Suppose $\Gamma$ contains a nontrivial finite order element $\gamma\neq e$, that is, $\Gamma$ is not torsion free. Then $[p_\gamma]$ generates a subgroup of rank one in $K_0(C^\ast(\Gamma))$ and any nonzero multiple of $[p_\gamma]$ is not in the image of the assembly map $\mu: K_0^\Gamma(E\Gamma)\to K_0(C^\ast(\Gamma))$ \cite[Theorem 2.3]{Weinberger:2013qy}.

\item Recall that if $M$ is a compact oriented topological manifold, the structure group $S(M)$ is the abelian group of equivalence classes of all pairs $(f, M')$, where $M'$ is a compact oriented manifold and $f: M' \to M$ is an orientation preserving homotopy equivalence (cf. \cite{MR2061749}). The rank of the abelian group $S(M)$ measures the degree of non-rigidity for $M$. Suppose $\dim M = 4k-1 \geq 5 $ and $\Gamma = \pi_1(M)$ is finitely embeddable into Hilbert space. Then the rank of the abelian group  $S(M)$ is $\geq N_\fin(\Gamma)$ \cite[Theorem 1.5]{Weinberger:2013qy}. 

\item Now suppose $M$ is a closed smooth spin manifold which supports at least one metric of positive scalar curvature. Following Stolz \cite{MR1818778, stolz-concordance}, Weinberger and Yu \cite[Section 4]{Weinberger:2013qy}  introduced an abelian group $P(M)$ of concordance classes of all positive scalar curvature metrics to measure the size of the space of all positive scalar curvature metrics (see also Section $\ref{sec:psc}$ below).  Suppose $\Gamma= \pi_1(M)$ is finitely embeddable into Hilbert space. If $\dim M = 2k-1\geq 5$, then the rank of the abelian group $P(M)$ is $\geq N_\fin(\Gamma)$. If $\dim M = 4k-1\geq 5$, then the rank of the abelian group $P(M)$ is  $\geq N_\fin(\Gamma) + 1$ \cite[Theorem 1.7]{Weinberger:2013qy}.
\end{enumerate}

\begin{remark}\label{rm:feh}
In fact, the same argument in \cite{Weinberger:2013qy} can be used to prove the following slightly stronger result. Let $\mathcal J_0^\fin(C^\ast(\Gamma))$ be the abelian subgroup of $K_0^\fin(C^\ast(\Gamma))$ generated by elements $[p_\gamma] - [p_\beta]$ with   $\ord(\gamma) = \ord(\beta)$. 
We define the reduced finite part of $K^0(C^\ast(\Gamma))$ to be 
\[ \rfK_0^\fin(C^\ast(\Gamma)) = K_0^\fin(C^\ast(\Gamma))/\mathcal J_0^\fin(C^\ast(\Gamma)). \]
Let  $\widetilde{\mathcal M}_{\gamma_1, \cdots, \gamma_n}$ be the image of ${\mathcal M}_{\gamma_1, \cdots, \gamma_n}$ in $\rfK_0^\fin(C^\ast(\Gamma))$. If $\Gamma$ is finitely embeddable into Hilbert space, then 
\begin{enumerate}[(1)]
\item  the abelian group $\widetilde{\mathcal M}_{\gamma_1, \cdots, \gamma_n}$ has rank $n$,
\item any nonzero element in $K_0^\fin(C^\ast(\Gamma))$ is not in the image of the assembly map 
\[ \mu: K_0^\Gamma(E\Gamma) \to K_0(C^\ast(\Gamma)), \]
where $E\Gamma$ is the universal space for proper and free $\Gamma$-action. 
\end{enumerate}
\end{remark}

So one is led to the following conjecture, which is a slight generalization of Conjecture $\ref{wyconj}$.

\begin{conjecture}\label{genconj} Let $\Gamma$ be a countable discrete group. 
Suppose $\{\gamma_1, \cdots, \gamma_n\}$ is a collection of elements in $\Gamma$ with distinct finite orders and $\gamma_i\neq e$ for all $1\leq i\leq n$. Then 
\begin{enumerate}[(1)]
\item  the abelian group $\widetilde{\mathcal M}_{\gamma_1, \cdots, \gamma_n}$ has rank $n$,
\item any nonzero element in $K_0^\fin(C^\ast(\Gamma))$ is not in the image of the assembly map 
\[ \mu: K_0^\Gamma(E\Gamma) \to K_0(C^\ast(\Gamma)), \]
where $E\Gamma$ is the universal space for proper and free $\Gamma$-action. 
\end{enumerate}
\end{conjecture}

\section{Finite embeddability and strongly finite embeddability}\label{sec:sfe}

In this section, we introduce the notion of strongly finitely embeddability for groups. Since we are interested in the fundamental groups of manifolds, all groups are assumed to be finitely generated in the following discussion.

Let $\Gamma$ be a countable discrete group. Then any set of $n$ automorphisms of $\Gamma$, say, $\psi_1, \cdots, \psi_n  \in \textup{Aut}(\Gamma) $, induces a natural action of $F_n$ the free group of $n$ generators on $\Gamma$. More precisely, if we denote the set of generators of $F_n$ by $\{s_1, \cdots, s_n\}$, then we have a homomorphism $F_n\to \textup{Aut}(\Gamma)$ by $s_i\mapsto \psi_i$. This homomorphism induces an action of $F_n$ on $\Gamma$. We denote by $\Gamma\rtimes_{\{\psi_1, \cdots, \psi_n\}} F_n$ the semi-direct product of $\Gamma$ and $F_n$ with respect to this action. If no confusion arises, we shall write $\Gamma\rtimes F_n$ instead of $\Gamma\rtimes_{\{\psi_1, \cdots, \psi_n\}} F_n$. 

\begin{definition}
A countable discrete group $\Gamma$ is said to be strongly finitely embeddable into Hilbert space $H$, if $\Gamma\rtimes_{\{\psi_1, \cdots, \psi_n\}} F_n$ is finitely embeddable into Hilbert space $H$ for all $n\in \mathbb N$ and all $\psi_1, \cdots, \psi_n  \in \textup{Aut}(\Gamma) $. 
\end{definition}

We remark that all coarsely embeddable groups are strongly finitely embeddable. Indeed, if a group $\Gamma$ is coarsely embeddable into Hilbert space, then $\Gamma\rtimes_{\{\psi_1, \cdots, \psi_n\}} F_n$ is coarsely embeddable (hence finitely embeddable) into Hilbert space for all $n\in \mathbb N$ and all $\psi_1, \cdots, \psi_n  \in \textup{Aut}(\Gamma) $.

 If a group $\Gamma$ has a torsion free normal subgroup $\Gamma^\prime$ such that $\Gamma/\Gamma^\prime$ is residually finite, then $\Gamma$ is strongly finitely embeddable into Hilbert space. Indeed, recall that any finitely generated group has only finitely many distinct subgroups of a given index. 
Let $\Gamma_m$ be the intersection of all subgroups of $\Gamma$ with index at most $m$. Then $\Gamma/\Gamma_m$ is a finite group. Moreover, for given $\psi_1, \cdots, \psi_n  \in \textup{Aut}(\Gamma) $, the induced action of $F_n$ on $\Gamma$ descends to an action of $F_n$ on $\Gamma/\Gamma_m$. In other words, we have a natural homomorphism 
\[ \phi_m\colon \Gamma\rtimes F_n \to (\Gamma/\Gamma_m)\rtimes G_m \]
where $G_m$ is the image of $F_n$ under the homomorphism $F_n\to \textup{Aut}(\Gamma/\Gamma_m)$.  It follows that, for any finite set $F\subseteq \Gamma$,  there exists a sufficiently large $m$ such that  the map  
\[ \phi_m: F\subset \Gamma\rtimes F_n \to (\Gamma/\Gamma_m)\rtimes G_m \]
satisfies 
\begin{enumerate}[(1)]
\item if $\gamma, \beta$ and $\gamma\beta$ are all in $F$, then $\phi(\gamma\beta) = \phi(\gamma) \phi(\beta)$;
\item if $\gamma$ is a finite order element\footnote{Note that in this case, all finite order elements in $\Gamma\rtimes_{\{\psi_1, \cdots, \psi_n\}} F_n$ come from $\Gamma$.} in $F$, then $\ord(\phi(\gamma)) = \ord(\gamma)$.
\end{enumerate}
Notice that $G_m$ is a finite group and  $(\Gamma/\Gamma_m)\rtimes G_m $ is coarsely embeddable into Hilbert space. This shows that $\Gamma$ is strongly finitely embeddable into Hilbert space.

To summarize, we see that the class of strongly finitely embeddable groups includes all residually finite groups, amenable groups, hyperbolic groups, and virtually torsion free groups (e.g. $Out(F_n)$). 

The notion of sofic groups is a generalization of amenable groups and residually finite groups. It is an open question whether sofic groups are  (strongly) finitely embeddable into Hilbert space. Narutaka Ozawa, Denise Osin and Thomas Delzant have independently constructed examples of groups which are not finitely embeddable into Hilbert space. An affirmative answer to the above question would imply that there exist non-sofic groups.

By definition, strongly finite embeddability implies finite embeddability. It is an open question whether the converse holds:

\begin{question*}
If a group is finitely embeddable into Hilbert space, then does it follow that the group is also  strongly finitely embeddable into Hilbert space?
\end{question*}

In fact, it was shown in \cite{Weinberger:2013qy} that Gromov's monster groups and any group of analytic diffeomorphisms of an analytic connected manifold fixing a given point are finitely embeddable into Hilbert space. It is still an open question whether these groups are strongly finitely embeddable into Hilbert space.  
 
\begin{remark}
The notion of strongly finite embeddability into Hilbert space can be generalized to strongly finite embeddability into Banach spaces with Property (H) \cite{MR2980001}. The main results of this paper remain true under this more flexible condition.
\end{remark}

\section{The moduli space of positive scalar curvature metrics} \label{sec:psc}

In this section, we prove the main results of the paper. We first recall the definition of the moduli group of positive scalar curvature metrics on a manifold \cite{Weinberger:2013qy}. Then we use the higher rho invariant and the finite part of $K$-theory of group $C^\ast$-algebra to give a lower bound on the rank of this moduli group.

Let $M$ be an oriented smooth closed manifold with $\dim M\geq 5$ and its fundamental group $\pi_1(M) = \Gamma$. Assume that $M$ carries a metric of positive scalar curvature. We denote it by $g_M$.  Let $I$ be the closed interval $[0, 1]$. Consider the connected sum $(M\times I)\sharp (M\times I)$, where the connected sum is performed away from the boundary of $M\times I$. Note that $\pi_1\big((M\times I)\sharp(M\times I)\big) = \Gamma\ast \Gamma$ the free product of two copies of $\Gamma$. 
\begin{definition}
We define the generalized connected sum $(M\times I)\natural(M\times I)$ to be the manifold obtained from $(M\times I)\sharp(M\times I)$ by removing the kernel of the homomorphism $\Gamma\ast \Gamma\to \Gamma$ through surgeries away from the boundary\footnote{The kernel of $\Gamma\ast \Gamma\to \Gamma$ is generated by loops which do not meet the boundary, and we may remove the kernel by surgeries away from the boundary.}. 
\end{definition}

Note that $(M\times I)\natural(M\times I)$ has four boundary components, two of them being $M$ and the other two being $-M$, where $-M$ is the manifold $M$ with its reversed orientation. 
Now suppose $g_1$ and $g_2$ are two positive scalar curvature metrics on $M$. We endow one boundary component $M$ with $g_M$, and endow the two $-M$ components with $g_1$ and $g_2$. Then by the Gromov-Lawson and Schoen-Yau surgery theorem for positive scalar curvature metrics \cite{MGBL80b, RSSY79b}, there exists a positive scalar curvature metric on $(M\times I)\natural (M\times I)$ which is a product metric near all boundary components. In particular, the restriction of this metric on the other boundary component $M$ has positive scalar curvature. We denote this metric on $M$ by $g$.     

\begin{definition}
Two positive scalar curvature metrics $g$ and $h$ on $M$ are concordant if there exists a positive scalar curvature metric  on $M\times I$ which is a product metric near the boundary and restricts to $g$ and $h$ on the two boundary components respectively. 
\end{definition}

One can in fact show that if $g$ and $g'$ are two positive scalar curvature metrics on $M$ obtained from the same pair of positive scalar curvature metrics $g_1$ and $g_2$ by the above procedure, then $g$ and $g'$ are concordant \cite[Section 4]{Weinberger:2013qy}. 

\begin{definition}
Fix a positive scalar curvature metric $g_M$ on $M$.  Let $P^+(M)$ be the set of all concordance classes of positive scalar metrics on $M$. Given $[g_1]$ and $[g_2]$ in $P^+(M)$, we define the sum of $[g_1]$ and $[g_2]$ (with respect to $[g_M]$) to be $[g]$ constructed from the procedure above. Then it is not difficult to verify that $P^+(M)$ becomes an abelian semigroup under this addition. We define the abelian group $P(M)$ to be the Grothendieck group of $P^+(M)$.  
\end{definition}

\begin{remark}
The abelian group $P(M)$ is defined for any closed smooth manifold with  $\dim M \geq 5$ which supports positive scalar curvature metrics. However, for our applications, we assume that $M$ is a spin manifold. 
\end{remark}

 The abelian group $P(M)$ measures the size of the space of positive scalar curvature metrics on $M$. Weinberger and Yu used the finite part of $K$-theory of $C^\ast(\pi_1(M))$ to give a lower bound of the rank of $P(M)$  \cite[Theorem 4.1]{Weinberger:2013qy}. 
More precisely, let $r_\fin(\Gamma)$ be the rank of the abelian subgroup of $K_0(C^\ast(\Gamma))$ generated by $[p_\gamma]$ for all finite order elements $\gamma\in \Gamma$. Note that here $\gamma$ is allowed to be the identity element $e$.  Suppose $M$ is a closed smooth spin manifold which supports positive scalar curvature metrics. Then one has
\begin{enumerate}[(1)]
\item if $\dim M = 2k-1 \geq 5$, then the rank of $P(M)$ is $\geq r_\fin(\Gamma) -1$; 
\item if $\dim M = 4k-1\geq 5 $, then the rank of $P(M)$ is $\geq r_\fin(\Gamma)$.
\end{enumerate}

One of main ingredients of their proof is the following proposition. For a finite group $F$, a $F$-manifold $Y$ is called $F$-connected if the quotient $Y/F$ is connected. Let $\mathbb Z_d$ be the cyclic group of order $d$. 
\begin{proposition}[{\cite[Proposition 4.4]{Weinberger:2013qy}}]\label{prop:gen}
Given positive integers $d$ and $k$,
\begin{enumerate}[(1)]
\item there exist $\mathbb Z_d$-connected closed spin $\mathbb Z_d$-manifolds $\{Y_1, \cdots, Y_n\}$ with $\dim Y_i = 2k$  such that 
\begin{enumerate}[$(a)$]
\item the $\mathbb Z_d$-equivariant indices of the Dirac operators on $\{Y_1, \cdots, Y_n\}$ rationally generate $RO(\mathbb Z_d)\otimes \mathbb Q$, modulo the regular representation,
\item $\mathbb Z_d$ acts on $Y_i$ freely except for finitely many fixed points,  
\end{enumerate} 
where $RO(\mathbb Z_d)$ is the real representation ring of $\mathbb Z_d$; 

\item there exist $\mathbb Z_d$-connected closed spin $\mathbb Z_d$-manifolds $\{Y_1, \cdots, Y_n\}$ with $\dim Y_i = 4k$ such that 
\begin{enumerate}[$(a)$]
\item the $\mathbb Z_d$-equivariant indices of the Dirac operators on $\{Y_1, \cdots, Y_n\}$ rationally generate $RO(\mathbb Z_d)\otimes \mathbb Q$,
\item $\mathbb Z_d$ acts on $Y_i$ freely except for finitely many fixed points.  
\end{enumerate} 

\end{enumerate}
\end{proposition}

Now let $M$ be a closed spin manifold with a positive scalar curvature metric $g_M$ and $\dim M\geq 5$  as before. This proposition, combined with the relative higher index theorem \cite{UB95, MR3122162}, allows one to construct  distinct concordance classes of positive scalar curvature metrics in $P(M)$. More precisely, for each nontrivial finite order element $\gamma\in \Gamma$, one can construct a new positive scalar curvature metric $h_\gamma$ on $M$ such that the relative higher index $\ind_\Gamma(g_M, h_\gamma) = [p_\gamma] \in K_0(C^\ast(\Gamma)) $, where $p_\gamma = \frac{1}{d}\sum_{k=1}^d \gamma^k$ with $d = \ord(\gamma)$. We refer the reader to \cite[Theorem 4.1]{Weinberger:2013qy} for the detailed construction. Here let us  recall the definition of this relative higher index $\ind_\Gamma(g_M, h_\gamma)$.  We endow $M\times \mathbb R$ with the metric $g_t + (dt)^2$ where  $g_t$ is a smooth path of Riemannian metrics on $M$ such that 
\[ 
g_t = \begin{cases}
g_M \quad \textup{for $t\leq 0$, }\\
h_\gamma \quad \textup{for $t\geq 1$,} \\
\textup{any smooth path of metrics from $g_M$ to $h_\gamma$ for $0\leq t \leq 1$.}
\end{cases}\] 
Then $M\times \mathbb R$ becomes a complete Riemannian manifold with positive scalar curvature away from a compact subset. Denote by $D_{M\times \mathbb R}$ the corresponding Dirac operator on $M\times \mathbb R$ with respect to this metric. Then the higher index of $D_{M\times \mathbb R}$  is well-defined and is denoted by $\ind_\Gamma(g_M, h_\gamma)$ (cf. \cite{MR3122162}).    

To summarize, one can construct distinct elements in $P(M)$ by surgery theory and the relative higher index theorem \cite{UB95,MR3122162}. Moreover, these elements are distinguished by their relative higher indices (with respect to $g_M$) \cite[Theorem 4.1]{Weinberger:2013qy}. 

In the following, we shall prove that these concordance classes of positive scalar curvature metrics remain distinct even after modulo diffeomorphisms. Moreover, our method of proof shows that these positive scalar curvature metrics constructed in \cite{Weinberger:2013qy} actually give distinct spin-bordism classes of positive scalar curvature metrics. We recall the definition of spin-bordism of positive scalar curvature metrics. 

\begin{definition}\label{def:sc}
Two metrics $g_1$ and $g_2$ are said to be spin-bordant if there exists a compact spin manifold $Y$ with a positive scalar curvature metric $g$ such that 
\begin{enumerate}[(1)]
\item the boundary of $Y$ is $\partial Y = M \amalg (-M)$, 
\item the metric $g$ is a product metric near the boundary, and $g$ restricts to $g_1$ (resp. $g_2$) on $M$ (resp. $-M$),
\item $Y$ admits a Galois $\Gamma$-covering $\widetilde Y$ which at the two boundary components restricts to the universal cover of $M$, where $\Gamma = \pi_1(M)$.   
\end{enumerate}  

\end{definition}

Now let us define the moduli group of positive scalar curvature metrics. We denote by $\diff(M)$ the group of all diffeomorphisms from $M$ to $M$. 

\begin{definition}
For each diffeomorphism $\psi \in \diff(M)$ and $g\in P(M)$, we denote by $\psi^\ast(g)$  the pullback metric of $g$ by $\psi$. We define $P_0(M)$ to be the abelian subgroup of $P(M)$ generated by elements of the form $[g] - [\psi^\ast(g)]$ for all $g\in P(M)$ and all $\psi\in \diff(M)$. Then the moduli group of positive scalar curvature metrics on $M$ is defined to be
\[ \widetilde P(M) = P(M)/P_0(M). \]
\end{definition}

\begin{remark}
Notice that, in general, the action of $\diff(M)$ on $P(M)$ is \textit{not} linear, but affine. 
\end{remark}

Note that every diffeomorphism $\psi\in \diff(M)$ induces an automorphism\footnote{To be precise, it only determines an outer automorphism of $\Gamma$. But all representatives in this outer automorphism class give rise to  the same automorphism of $K_i(C^\ast(\Gamma))$.} of $\Gamma = \pi_1(M)$. We denote this automorphism of $\Gamma$ also by $\psi$.

Recall that $N_\fin(\Gamma)$ is the the cardinality of the following collection of positive integers:
\[ \{ d\in \mathbb N_+\mid \exists \gamma\in \Gamma \ \textup{such that} \ \gamma\neq e \ \textup{and}\ \ord(\gamma) = d \}. \]
Then we have the following main theorem of the paper. 
\begin{theorem}\label{thm:psc}
Let $M$ be a closed spin manifold  with $\dim M = 2k+1 \geq 5$ which carries a positive scalar curvature metric. If $\Gamma =  \pi_1(M)$ is strongly finitely embeddable into Hilbert space, then the rank of the abelian group  $\widetilde P(M)$ is  $\geq N_\fin(\Gamma)$.
\end{theorem}

If we denote the space of all positive scalar curvature metrics on $M$ by $\mathcal R^+(M)$, then  $\diff(M)$ again acts on $\mathcal R^+(M)$ by pulling back the metrics. The moduli space of all positive scalar curvature metrics on $M$ is defined to be $\mathcal R^+(M)/\diff(M)$ the quotient space of $\mathcal R^+(M)$ by $\diff(M)$. Piazza and Schick used a different method to show the moduli space $\mathcal R^+(M)/\diff(M)$ has infinitely many connected components when $\dim M = 4k+3 \geq 5$ and the fundamental group $\pi_1(M)$ is not torsion free \cite{MR2366359}. The following result is a refinement of their theorem. 

\begin{theorem}\label{thm:tor}
Let $M$ be a closed spin manifold which carries a positive scalar curvature metric. Suppose $\dim M = 2k+1\geq 5$ and $\Gamma = \pi_1(M)$ is not torsion free. Then the rank of the abelian group $\widetilde P(M)$ is $\geq 1$.
\end{theorem}
\begin{proof}
Recall that for any non-torsion-free countable discrete group $G$, if $\gamma \neq e$ is a finite order element of $G$, then $[p_\gamma]$ generates a subgroup of rank one in $K_0(C^\ast(G))$ and any nonzero multiple of $[p_\gamma]$ is not in the image of the assembly map $\mu\colon K_0^\Gamma(EG)\to K_0(C^\ast(G))$ \cite[Theorem 2.3]{Weinberger:2013qy}. Now we apply this fact to $\Gamma\rtimes F_n$, and the theorem is proved the same way as Theorem $\ref{thm:psc}$.  
\end{proof}

The following result is an immediate corollary of Theorem $\ref{thm:tor}$.

\begin{corollary}\label{cor:tor}
Let $M$ be a closed spin manifold with $\dim M = 2k+1 \geq 5$ which carries a positive scalar curvature metric. If $\Gamma =  \pi_1(M)$ is not torsion free,  then the moduli space $\mathcal R^+(M)/\diff(M)$ has infinitely many connected components, i.e.,
\[ |\pi_0(\mathcal R^+(M)/\diff(M))| = \infty.\]
\end{corollary}

Now let us give the proof of Theorem $\ref{thm:psc}$.

\begin{proof}[Proof of Theorem $\ref{thm:psc}$]

Consider the following short exact sequence
\[  0\to C^\ast_{L,0}(\widetilde M)^\Gamma \to C^\ast_L(\widetilde M)^\Gamma \to C^\ast(\widetilde M)^\Gamma \to 0\]
where $\widetilde M$ is the universal cover of $M$. It induces the following six-term long exact sequence:
\[ \xymatrix{K_0(C_{L,0}^\ast(\widetilde M)^\Gamma)\ar[r] &  K_0(C_L^\ast(\widetilde M)^\Gamma)  \ar[r]^{\mu_M} &  K_0(C^\ast(\widetilde M)^\Gamma) \ar[d]^{\partial} \\
K_1(C^\ast(\widetilde M)^\Gamma) \ar[u]& K_{1}(C^\ast_{L}(\widetilde M)^\Gamma) \ar[l] &  K_{1}(C_{L,0}^\ast(\widetilde M)^\Gamma) \ar[l]} \]
Recall that we have $ K_0(C_L^\ast(\widetilde M)^\Gamma) \cong K_0^\Gamma(\widetilde M) $ and $K_0(C^\ast(\widetilde M)^\Gamma) \cong K_0(C^\ast(\Gamma))$.

Fix a positive scalar curvature metric $g_M$ on $M$. For each finite order element $\gamma\in \Gamma$, we can construct a new positive scalar curvature metric $h_\gamma$ on $M$ such that the relative higher index $\ind_\Gamma(g_M, h_\gamma) = [p_\gamma] \in K_0(C^\ast(\Gamma)) $ \cite[Theorem 4.1]{Weinberger:2013qy}. Let us still denote by $h_\gamma$ (resp. $g_M$) the metric on $\widetilde M$ lifted from the metric $h_\gamma$ (resp. $g_M$) on $M$. Let $\rho(D, h_\gamma)$ and $\rho(D, g_M)$ be the higher rho invariants for the pairs $(D, h_\gamma)$ and $(D, g_M)$, where $D$ is the Dirac operator on $\widetilde M$ (cf. Section $\ref{sec:rho}$). 
Then we have
\begin{equation}\label{eq:bd}
\partial([p_\gamma]) = \partial\big(\ind_\Gamma(g_M, h_\gamma)\big) = \rho(D, h_\gamma) - \rho(D, g_M),
\end{equation}
(cf.  \cite[Theorem 1.17]{Piazza:2012fk}, \cite[Theorem 4.1]{Xie2014823}).

One of the key points of the proof is to construct a certain group homomorphism on $\widetilde P(M)$ which can be used to distinguish elements in $\widetilde P(M)$. First, we define a map $ \varrho: P(M) \to  K_{1}(C_{L,0}^\ast(\widetilde M)^\Gamma)$  by 
\[\varrho(h) := \rho(D, h) - \rho(D, g_M)\]
 for all $h\in P(M)$. It follows from the definition of $P(M)$ and  \cite[Theorem 4.1]{Xie2014823} that the map $\varrho$ is a well-defined group homomorphism. Now recall that a diffeomorphism $\psi \in \diff(M)$ induces a homomorphism
\[\psi_\ast: K_{1}(C_{L,0}^\ast(\widetilde M)^\Gamma)\to K_{1}(C_{L,0}^\ast(\widetilde M)^\Gamma).\]
Let $\mathcal I_1(C_{L,0}^\ast(\widetilde M)^\Gamma)$ be the subgroup of $K_{1}(C_{L,0}^\ast(\widetilde M)^\Gamma)$  generated by elements of the form $[x] - \psi_\ast[x]$ for all $[x]\in K_{1}(C_{L,0}^\ast(\widetilde M)^\Gamma)$ and all $\psi\in \diff(M)$. We see that $\varrho$ descends to a group homomorphism
\[ \widetilde \varrho: \widetilde P(M) \to K_{1}(C_{L,0}^\ast(\widetilde M)^\Gamma)\big/\mathcal I_1(C_{L,0}^\ast(\widetilde M)^\Gamma). \]
To see this, it suffices to verify that
\[  \varrho(h) - \varrho(\psi^\ast(h)) \in \mathcal I_1(C_{L,0}^\ast(\widetilde M)^\Gamma) \]
for all $[h]\in P(M)$ and $\psi\in \diff(M)$. Indeed, we have
\begin{align*}
 \varrho(h) - \varrho(\psi^\ast(h)) & = \rho(D, h) - \rho(D, g_M) - \big (\rho(D, \psi^\ast (h)) - \rho(D, g_M) \big) \\ 
 & = \rho(D, h) - \rho(D, \psi^\ast (h)) \\
 & = \rho(D, h) - \psi_\ast( \rho(D, h)) \in \mathcal I_1(C_{L,0}^\ast(\widetilde M)^\Gamma).
\end{align*}

We remark that it is crucial to use the higher rho invariant, instead of the relative higher index, to construct such a group homomorphism.  Let us explain the subtlety here.  Note that there is in fact a well-defined group homomorphism 
$ \ind_{rel}: P(M) \to  K_{0}(C^\ast(\Gamma))$ by $ \ind_{rel} (h)  = \ind_\Gamma(D; g_M, h).$
The well-definedness of $\ind_{rel}$ follows from the definition of $P(M)$ and the relative higher index theorem \cite{UB95}\cite{MR3122162}. 
However, in general, it is \textit{not} clear at all whether $\ind_{rel}$ descends to a group homomorphism $\widetilde P(M) \to  K_{0}(C^\ast(\Gamma))/\mathcal I_0(C^\ast(\Gamma)),$ 
where 	$\mathcal I_0(C^\ast(\Gamma))$ is the subgroup of $K_{0}(C^\ast(\Gamma))$  generated by elements of the form $[x] - \psi_\ast[x]$ for all $[x]\in K_{0}(C^\ast(\Gamma))$ and all $\psi\in \diff(M)$.

Now for a collection of elements $\{\gamma_1, \cdots, \gamma_n\}$ with distinct finite orders, we consider the associated collection of positive scalar curvature metrics $\{h_{\gamma_1}, \cdots, h_{\gamma_n}\}$ as before. To prove the theorem, it suffices to show that for any collection of elements $\{\gamma_1, \cdots, \gamma_n\}$ with distinct finite orders, the elements 
\[ \widetilde \varrho(h_{\gamma_1}), \cdots, \widetilde \varrho(h_{\gamma_n})\]  
are linearly independent in $K_{1}(C_{L,0}^\ast(\widetilde M)^\Gamma)\big/\mathcal I_1(C_{L,0}^\ast(\widetilde M)^\Gamma)$. 

Let us assume the contrary, that is, there exist $[x_1], \cdots, [x_m]\in K_{1}(C_{L,0}^\ast(\widetilde M)^\Gamma) $ and $\psi_1, \cdots, \psi_m\in \diff(M)$ such that
\begin{equation}\label{eq:van}
\sum_{i=1}^n c_i \varrho(h_{\gamma_i}) = \sum_{j=1}^{m} \big([x_j] - (\psi_{j})_\ast[x_j]\big),
\end{equation}
where $c_1, \cdots, c_n\in \mathbb Z$ with at least one $c_i\neq 0$.

We denote by $W$ the wedge sum of $m$ circles. The fundamental group $\pi_1(W)$ is the free group $F_m$ of $m$ generators $\{s_1, \cdots, s_m\}$. We denote the universal cover of $W$ by $\widetilde W$. Clearly, $\widetilde W$ is  the Cayley graph of $F_m$ with respect to the generating set $\{s_1, \cdots, s_m, s_1^{-1}, \cdots, s_m^{-1}\}$. Notice that $F_m$ acts on $M$ through the diffeomorphisms $\psi_1, \cdots, \psi_m$. In other words, we have a homomorphism  $F_m \to \diff(M)$ by $s_i\mapsto \psi_i$. We define
\[  X = M \times_{F_m}\widetilde W. \]
Notice that $\pi_1(X) = \Gamma\rtimes_{\{\psi_1, \cdots, \psi_m\}} F_m$. Let us  write $\Gamma\rtimes F_m$ for $\Gamma\rtimes_{\{\psi_1, \cdots, \psi_m\}} F_m$, if no confusion arises.

Let $\widetilde X$ be the universal cover of $X$.  We have the following short exact sequence: 
\[  0\to C^\ast_{L,0}(\widetilde X)^{\Gamma\rtimes F_m} \to C^\ast_L(\widetilde X)^{\Gamma\rtimes F_m} \to C^\ast(\widetilde X)^{\Gamma\rtimes F_m} \to 0.\]
 Recall that  $ K_0(C_L^\ast(\widetilde X)^{\Gamma\rtimes F_m}) \cong K_0^{\Gamma\rtimes F_m}(\widetilde X) $ and $K_0(C^\ast(\widetilde X)^{\Gamma\rtimes F_m}) \cong K_0(C^\ast(\Gamma\rtimes F_m))$. So we have the following  six-term long exact sequence:
\begin{equation}\label{cd:bc}
\begin{gathered}
\xymatrix{K_0(C_{L,0}^\ast(\widetilde X)^{\Gamma\rtimes F_m})\ar[r] &  K_0^{\Gamma\rtimes F_m}(\widetilde X)  \ar[r] &  K_0(C^\ast(\Gamma\rtimes F_m)) \ar[d]^{\partial} \\
K_1(C^\ast(\Gamma\rtimes F_m)) \ar[u]& K_{1}^{\Gamma\rtimes F_m}(\widetilde X) \ar[l] &  K_{1}(C_{L,0}^\ast(\widetilde X)^{\Gamma\rtimes F_m}) \ar[l]} 
\end{gathered}
\end{equation}
Now recall the following Pimsner-Voiculescu exact sequence \cite{MR670181}:
\[ \xymatrixcolsep{6pc} \xymatrix{ \bigoplus_{j=1}^m K_0(C^\ast(\Gamma))\ar[r]^-{\sum_{j=1}^m 1-(\psi_j)_\ast} &   K_0(C^\ast(\Gamma))  \ar[r]^{i_\ast} &  K_0(C^\ast(\Gamma\rtimes F_m)) \ar[d] \\
K_1(C^\ast(\Gamma\rtimes F_m)) \ar[u]& K_{1}(C^\ast(\Gamma)) \ar[l] &  \bigoplus_{j=1}^m K_{1}(C^\ast(\Gamma)) \ar[l]_-{\sum_{j=1}^m 1-(\psi_j)_\ast}
} \]
where $(\psi_j)_\ast$ is induced by $\psi_j$ and $i_\ast$ is induced by the inclusion map of $\Gamma$ into $\Gamma\rtimes F_m$. Similarly, we also have the following two Pimsner-Voiculescu type exact sequences for $K$-homology and the $K$-theory groups of $C_{L,0}^\ast$-algebras in the diagram $\eqref{cd:bc}$ above.
\[ \xymatrixcolsep{6pc} \xymatrix{\bigoplus_{i=1}^m K_0^\Gamma(\widetilde M)\ar[r]^-{\sum_{j=1}^m 1-(\psi_j)_\ast} &  K_0^\Gamma(\widetilde M)  \ar[r]^{i_\ast} &  K_0^{\Gamma\rtimes F_m}(\widetilde X) \ar[d] \\
K_1^{\Gamma\rtimes F_m}(\widetilde X) \ar[u]& K_{1}^\Gamma(\widetilde M) \ar[l]&  \bigoplus_{i=1}^m K_{1}^\Gamma(\widetilde M) \ar[l]_-{\sum_{j=1}^m 1-(\psi_j)_\ast}  } \]
\[\xymatrixcolsep{5pc} \xymatrix{\bigoplus_{i=1}^m K_0(C_{L,0}^\ast(\widetilde M)^\Gamma)\ar[r]^-{\sum_{j=1}^m 1-(\psi_j)_\ast}  &  K_0(C_{L,0}^\ast(\widetilde M)^\Gamma)  \ar[r]^{i_\ast} &  K_0(C_{L,0}^\ast(\widetilde X)^{\Gamma\rtimes F_m}) \ar[d] \\
K_1(C_{L,0}^\ast(\widetilde X)^{\Gamma\rtimes F_m}) \ar[u]& K_{1}(C_{L,0}^\ast(\widetilde M)^\Gamma) \ar[l] & \bigoplus_{i=1}^m  K_{1}(C_{L,0}^\ast(\widetilde M)^\Gamma) \ar[l]_-{\sum_{j=1}^m 1-(\psi_j)_\ast} } \]
where again $(\psi_j)_\ast$ and $i_\ast$ are defined in the obvious way.

Combining these Pimsner-Voiculescu exact sequences together, we have the following commutative diagram: 
\begin{equation}\label{cd:pv}
\begin{gathered}
\scalebox{1}{ \xymatrix{  
& \vdots\ar[d]  & \vdots \ar[d] & \vdots\ar[d] & \\
 \ar[r] & \bigoplus_{j=1}^m K_0^{\Gamma}(\widetilde M) \ar[d] \ar[r]^-\sigma & K_0^{\Gamma}(\widetilde M) \ar[d] \ar[r]^-{i_\ast}  & K_0^{\Gamma\rtimes F_m}(\widetilde X) \ar[d]^{\mu} \ar[r] & \\
 \ar[r] &  \bigoplus_{j=1}^m K_0(C^\ast(\Gamma))\ar[r]^-{\sigma} \ar[d] &  K_0(C^\ast(\Gamma))  \ar[r]^-{i_\ast} \ar[d] &  K_0(C^\ast(\Gamma\rtimes F_m)) \ar[d]^{\partial_{\Gamma\rtimes F_m}} \ar[r] & \\
 \ar[r] &  \bigoplus_{j=1}^m K_{1}(C_{L,0}^\ast(\widetilde M)^{\Gamma}) \ar[r]^-\sigma \ar[d] & K_{1}(C_{L,0}^\ast(\widetilde M)^{\Gamma}) \ar[r]^-{i_\ast} \ar[d] & K_{1}(C_{L,0}^\ast(\widetilde X)^{\Gamma\rtimes F_m}) \ar[r] \ar[d] &  \\
& \vdots  & \vdots  & \vdots & \\
} }
\end{gathered}
\end{equation}
where $\sigma = \sum_{j=1}^m 1-(\psi_j)_\ast$. Notice that all rows and  columns are exact. 

Now on one hand,  if we pass Equation $\eqref{eq:van}$ to $K_{1}(C_{L,0}^\ast(\widetilde X)^{\Gamma\rtimes F_m})$ under the map $i_\ast$, then it follows immediately that 
\[  \sum_{k=1}^n c_k\cdot i_\ast[\varrho(h_{\gamma_k})] = 0 \textup{ in $K_{1}(C_{L,0}^\ast(\widetilde X)^{\Gamma\rtimes F_m})$}  , \]
where at least one $c_k\neq 0$. On the other hand, by assumption, $\Gamma$ is strongly finitely embeddable into Hilbert space. Hence $\Gamma\rtimes F_m$ is finitely embeddable into Hilbert space. Therefore, Conjecture $\ref{wyconj}$ holds for $\Gamma\rtimes F_m$, which implies the following.
\begin{enumerate}[(i)]
\item $\{[p_{\gamma_1}], \cdots, [p_{\gamma_n}]\}$ generates a rank $n$ abelian subgroup of $K_0^\fin(C^\ast(\Gamma\rtimes F_m))$, since $\gamma_1, \cdots, \gamma_n$ have distinct finite orders. In other words, 
\[ \sum_{k=1}^n c_k [p_{\gamma_k}] \neq 0  \in K_0^\fin(C^\ast(\Gamma\rtimes F_m))\]
if at least one $c_k\neq 0$.
\item Every nonzero element in $K_0^\fin(C^\ast(\Gamma\rtimes F_m))$ is not in the image of the assembly map 
\[ \mu: K_0^{\Gamma\rtimes F_m}(E(\Gamma\rtimes F_m)) \to K_0(C^\ast(\Gamma\rtimes F_m)), \]
where $E(\Gamma\rtimes F_m)$ is the universal space for proper and free $\Gamma\rtimes F_m$-action. In particular, every nonzero element in $K_0^\fin(C^\ast(\Gamma\rtimes F_m))$ is not in the image of the map 
\[ \mu: K_0^{\Gamma\rtimes F_m}(\widetilde X) \to K_0(C^\ast(\Gamma\rtimes F_m)) \]
in diagram $\eqref{cd:pv}$. It follows that the map 
\[ \partial_{\Gamma\rtimes F_m}: K^\fin_0(C^\ast(\Gamma\rtimes F_m)) \to K_{1}(C_{L,0}^\ast(\widetilde X)^{\Gamma\rtimes F_m}) \]
is injective. In other words, $\partial_{\Gamma\rtimes F_m}$ maps a nonzero element in $ K^\fin_0(C^\ast(\Gamma\rtimes F_m))$ to a nonzero element in $K_{1}(C_{L,0}^\ast(\widetilde X)^{\Gamma\rtimes F_m}) $. 
\end{enumerate} 
To summarize, we have 
\begin{enumerate}[(a)]
\item $\sum_{k=1}^n c_k [p_{\gamma_k}] \neq 0$ in $K^\fin_0(C^\ast(\Gamma\rtimes F_m))$,
\item  $ \sum_{k=1}^n c_k \cdot i_\ast[\varrho(h_{\gamma_k})] = 0 $ in $K_{1}(C_{L,0}^\ast(\widetilde X)^{\Gamma\rtimes F_m})$, 
\item the map $ \partial_{\Gamma\rtimes F_m}: K^\fin_0(C^\ast(\Gamma\rtimes F_m)) \to K_{1}(C_{L,0}^\ast(\widetilde X)^{\Gamma\rtimes F_m})$ is injective,
\item and by Equation $\eqref{eq:bd}$,  $ \partial_{\Gamma\rtimes F_m} \Big( \sum_{k=1}^n c_k [p_{\gamma_k}] \Big) =  \sum_{k=1}^n c_k \cdot i_\ast[\varrho(h_{\gamma_k})]$.
\end{enumerate}
Therefore, we arrive at a contradiction. This finishes the proof.

\end{proof}

\begin{remark}
In the above proof, we used the higher rho invariant to show that a particular collection of elements in $\widetilde P(M)$ are linearly independent. Recall that the higher rho invariant is constant within a spin-bordism class of positive scalar curvature metrics, cf. \cite[Theorem 4.1]{Xie2014823}. Therefore, our proof in fact shows that this particular collection of concordance classes of positive scalar curvature metrics give rise to distinct spin-bordism classes of positive scalar curvature metrics. 
\end{remark}

\begin{remark}
All the above discussion also applies to the case where $M$ is even dimensional,  with an additional assumption on its fundamental group $\pi_1(M)$. More precisely, the following analogues of Theorem $\ref{thm:psc}$, Theorem $\ref{thm:tor}$ and Corollary $\ref{cor:tor}$ hold. Throughout this remark, let us assume that  $M$ is a closed spin manifold  with $\dim M = 2k \geq 5$ which carries a positive scalar curvature metric, and there is a surjective group homomorphism $\phi: \Gamma = \pi_1(M) \twoheadrightarrow \mathbb Z\times G$ for some group $G$.
\begin{enumerate}[(i)]
\item If $G$ is strongly finite embeddable, then the rank of $\widetilde P(M)$ is  $\geq N_\fin(G)$. 
\item If $G$ is not torsion-free, then the rank of $\widetilde P(M)$ is  $\geq 1$. In particular, this implies that  $ |\pi_0(\mathcal R^+(M)/\diff(M))| = \infty.$
\end{enumerate}
The proofs of these statements are essentially the same as those in the odd case, combined with a suspension argument. More precisely, we have the following observations.   
\begin{enumerate}[(a)]
\item The surjection $\phi: \Gamma = \pi_1(M) \twoheadrightarrow \mathbb Z\times G$ determines a  $(\mathbb Z\times G)$-covering of $M$, denoted by $M_{\mathbb Z\times G}$. We work with $M_{\mathbb Z\times G}$ instead of the universal covering $\widetilde M$. 
\item For each nontrivial torsion element $\gamma\in G$, let $p_\gamma = \sum_{k=1}^{d} \gamma^k$ with $d = \ord(\gamma)$. We fix a positive scalar curvature metric $g_M$ on $M$ and use the manifold $S^1\times Y$ to construct new positive scalar metrics $h_\gamma$ on $M$, where $Y$ is a manifold as in Proposition $\ref{prop:gen}$. \textit{Notice that the suspension is applied to the manifold $Y$, not the manifold $M$.} Recall that 
 \[ K_1(C^\ast(\mathbb Z\times G)) = \big(K_1(C^\ast(\mathbb Z))\otimes K_0(C^\ast(G)) \big) \oplus \big( K_0(C^\ast(\mathbb Z))\otimes K_1(C^\ast(G)) \big). \] Let $D_{S^1}$ be the standard Dirac operator on $S^1$ and $u = \ind_{\mathbb Z}(D_{S^1})\in  K_1(C^\ast(\mathbb Z))$. Note that $u$ is a generator of $K_1(C^\ast(\mathbb Z)) \cong \mathbb Z$. Moreover, we have 
 \[ \ind_{\mathbb Z\times G} (g_M, h_\gamma)  = [u]\otimes [p_\gamma] \in K_1(C^\ast(\mathbb Z))\otimes K_0(C^\ast(G)).  \]
Recall that $p_\gamma$ is not in the image of the assembly map (for free actions) $\mu: K_0^{G}(EG) \to K_0(C^\ast(G))$. It follows that $[u]\otimes [p_\gamma]$ is not in the image of the assembly map (for free actions) $\mu: K_0^{\mathbb Z\times G}(E(\mathbb Z\times G)) \to K_0(C^\ast(\mathbb Z\times G))$. This argument also applies to the case of linear combinations of $[p_{\gamma_i}]$ for a collection of elements  $\{\gamma_1, \cdots, \gamma_n\}$ with distinct finite orders in $G$.

\item Various terms in the proof for the odd case above now switch parities. For example, the higher rho invariant $\rho(D_M, h_\gamma)$ now lies in $K_0(C_{L,0}^\ast(M_{\mathbb Z\times G})^{\mathbb Z\times G})$.

\end{enumerate}

\end{remark}


\begin{thebibliography}{10}
	
	\bibitem{MR1339924}
	B.~Botvinnik and P.~B. Gilkey.
	\newblock The eta invariant and metrics of positive scalar curvature.
	\newblock {\em Math. Ann.}, 302(3):507--517, 1995.
	
	\bibitem{UB95}
	U.~Bunke.
	\newblock A {$K$}-theoretic relative index theorem and {C}allias-type {D}irac
	operators.
	\newblock {\em Math. Ann.}, 303(2):241--279, 1995.
	
	\bibitem{MR806699}
	J.~Cheeger and M.~Gromov.
	\newblock Bounds on the von {N}eumann dimension of {$L^2$}-cohomology and the
	{G}auss-{B}onnet theorem for open manifolds.
	\newblock {\em J. Differential Geom.}, 21(1):1--34, 1985.
	
	\bibitem{MR2431253}
	G.~Gong, Q.~Wang, and G.~Yu.
	\newblock Geometrization of the strong {N}ovikov conjecture for residually
	finite groups.
	\newblock {\em J. Reine Angew. Math.}, 621:159--189, 2008.
	
	\bibitem{Gong:2013fk}
	S.~Gong.
	\newblock Finite part of operator {$K$}-theory for groups with rapid decay.
	\newblock arxiv.org/abs/1309.3341, 2013.
	
	\bibitem{MR1616436}
	R.~I. Grigorchuk.
	\newblock An example of a finitely presented amenable group that does not
	belong to the class {EG}.
	\newblock {\em Mat. Sb.}, 189(1):79--100, 1998.
	
	\bibitem{MGBL80b}
	M.~Gromov and H.~B. Lawson, Jr.
	\newblock The classification of simply connected manifolds of positive scalar
	curvature.
	\newblock {\em Ann. of Math. (2)}, 111(3):423--434, 1980.
	
	\bibitem{Hanke:2012fk}
	B.~Hanke, T.~Schick, and W.~Steimle.
	\newblock The space of metrics of positive scalar curvature.
	\newblock 2012.
	
	\bibitem{MR2761858}
	N.~Higson and J.~Roe.
	\newblock {$K$}-homology, assembly and rigidity theorems for relative eta
	invariants.
	\newblock {\em Pure Appl. Math. Q.}, 6(2, Special Issue: In honor of Michael
	Atiyah and Isadore Singer):555--601, 2010.
	
	\bibitem{NH74}
	N.~Hitchin.
	\newblock Harmonic spinors.
	\newblock {\em Adv. Math.}, 14:1--55, 1974.
	
	\bibitem{GK88}
	G.~Kasparov.
	\newblock Equivariant {$KK$}-theory and the {N}ovikov conjecture.
	\newblock {\em Invent. Math.}, 91(1):147--201, 1988.
	
	\bibitem{MR2980001}
	G.~Kasparov and G.~Yu.
	\newblock The {N}ovikov conjecture and geometry of {B}anach spaces.
	\newblock {\em Geom. Topol.}, 16(3):1859--1880, 2012.
	
	\bibitem{ELPP01}
	E.~Leichtnam and P.~Piazza.
	\newblock On higher eta-invariants and metrics of positive scalar curvature.
	\newblock {\em $K$-Theory}, 24(4):341--359, 2001.
	
	\bibitem{MR1726745}
	J.~Lott.
	\newblock Delocalized {$L^2$}-invariants.
	\newblock {\em J. Funct. Anal.}, 169(1):1--31, 1999.
	
	\bibitem{MR2366359}
	P.~Piazza and T.~Schick.
	\newblock Groups with torsion, bordism and rho invariants.
	\newblock {\em Pacific J. Math.}, 232(2):355--378, 2007.
	
	\bibitem{Piazza:2012fk}
	P.~Piazza and T.~Schick.
	\newblock Rho-classes, index theory and {S}tolz's positive scalar curvature
	sequence.
	\newblock preprint, 2012.
	
	\bibitem{MR670181}
	M.~Pimsner and D.~Voiculescu.
	\newblock {$K$}-groups of reduced crossed products by free groups.
	\newblock {\em J. Operator Theory}, 8(1):131--156, 1982.
	
	\bibitem{MR2061749}
	A.~Ranicki.
	\newblock {\em Algebraic and geometric surgery}.
	\newblock Oxford Mathematical Monographs. The Clarendon Press Oxford University
	Press, Oxford, 2002.
	\newblock Oxford Science Publications.
	
	\bibitem{MR1147350}
	J.~Roe.
	\newblock Coarse cohomology and index theory on complete {R}iemannian
	manifolds.
	\newblock {\em Mem. Amer. Math. Soc.}, 104(497):x+90, 1993.
	
	\bibitem{MR1268192}
	J.~Rosenberg and S.~Stolz.
	\newblock Manifolds of positive scalar curvature.
	\newblock In {\em Algebraic topology and its applications}, volume~27 of {\em
		Math. Sci. Res. Inst. Publ.}, pages 241--267. Springer, New York, 1994.
	
	\bibitem{MR1818778}
	J.~Rosenberg and S.~Stolz.
	\newblock Metrics of positive scalar curvature and connections with surgery.
	\newblock In {\em Surveys on surgery theory, {V}ol. 2}, volume 149 of {\em Ann.
		of Math. Stud.}, pages 353--386. Princeton Univ. Press, Princeton, NJ, 2001.
	
	\bibitem{RSSY79b}
	R.~Schoen and S.~T. Yau.
	\newblock On the structure of manifolds with positive scalar curvature.
	\newblock {\em Manuscripta Math.}, 28(1-3):159--183, 1979.
	
	\bibitem{stolz-concordance}
	S.~Stolz.
	\newblock Concordance classes of positive scalar curvature metrics.
	\newblock preprint.
	
	\bibitem{SS95}
	S.~Stolz.
	\newblock Positive scalar curvature metrics---existence and classification
	questions.
	\newblock In {\em Proceedings of the {I}nternational {C}ongress of
		{M}athematicians, {V}ol.\ 1, 2 ({Z}{\"u}rich, 1994)}, pages 625--636, Basel,
	1995. Birkh{\"a}user.
	
	\bibitem{Weinberger:2013qy}
	S.~Weinberger and G.~Yu.
	\newblock Finite part of operator {K}-theory for groups finitely embeddable
	into {H}ilbert space and the degree of non-rigidity of manifolds.
	\newblock arXiv:1308.4744, 2013.
	
	\bibitem{Xie2014823}
	Z.~Xie and G.~Yu.
	\newblock Positive scalar curvature, higher rho invariants and localization
	algebras.
	\newblock {\em Adv. Math.}, 262:823--866, 2014.
	
	\bibitem{MR3122162}
	Z.~Xie and G.~Yu.
	\newblock A relative higher index theorem, diffeomorphisms and positive scalar
	curvature.
	\newblock {\em Adv. Math.}, 250:35--73, 2014.
	
	\bibitem{MR1451759}
	G.~Yu.
	\newblock Localization algebras and the coarse {B}aum-{C}onnes conjecture.
	\newblock {\em $K$-Theory}, 11(4):307--318, 1997.
	
	\bibitem{GY00}
	G.~Yu.
	\newblock The coarse {B}aum-{C}onnes conjecture for spaces which admit a
	uniform embedding into {H}ilbert space.
	\newblock {\em Invent. Math.}, 139(1):201--240, 2000.
	
	\bibitem{MR2732068}
	G.~Yu.
	\newblock A characterization of the image of the {B}aum-{C}onnes map.
	\newblock In {\em Quanta of Maths}, volume~11 of {\em Clay Math. Proc.}, pages
	649--657. Amer. Math. Soc., Providence, RI, 2010.
	
\end{thebibliography}

\end{document}